\documentclass[12pt]{amsart}
\usepackage[english]{babel}
\usepackage[utf8]{inputenc}
\usepackage{amsmath}
\usepackage{amsthm}
\usepackage{graphicx}
\usepackage{amssymb}
\usepackage{enumerate}
\usepackage{enumitem}
\usepackage{color}
\usepackage{accents}

\newtheorem{theorem}{Theorem}
\newtheorem{corollary}{Corollary}
\newtheorem{lemma}{Lemma}
\newtheorem{proposition}{Proposition}
\newtheorem{definition}{Definition}
\newtheorem{remark}{Remark}

\newcommand{\D}{\mathbb{D}}

\newcommand{\eit}{\mathrm{e}^{i\theta}}

\newcommand{\eps}{\varepsilon}
\newcommand{\supp}{\mathrm{supp\;}}
\newcommand{\dist}{\mathrm{dist\,}}
\newcommand{\card}{\mathrm{Card\;}}

\begin{document}

\title{Wepable Inner Functions}

\author{Alexander Borichev, Artur Nicolau
and
Pascal J. Thomas}

\address{A. Borichev: Aix Marseille Universit\'e\\
CNRS\\ Centrale Marseille\\ I2M\\ 13453 Marseille\\ France}
\email{alexander.borichev@math.cnrs.fr}

\address{A. Nicolau: Departament de Matem\'atiques\\
Universitat Aut\`onoma de Barcelona\\ 08193 Barcelona\\ Spain}
\email{artur@mat.uab.cat}

\address{P.J. Thomas: Universit\'e de Toulouse\\ UPS, INSA, UT1, UTM \\
Institut de Math\'e\-ma\-tiques de Toulouse\\
F-31062 Toulouse, France} 
\email{pascal.thomas@math.univ-toulouse.fr}



\thanks{Second author was supported in part by the MINECO grants MTM2011-24606, MTM2014-51824-P and by 2014SGR 75, Generalitat de Catalunya}

\date{\today}

\maketitle

\hfill{To Nikolai Nikolski on occasion of his birthday}

\begin{abstract}
Following Gorkin, Mortini, and Nikolski, we say that an inner function $I$ in $H^\infty(\D)$
has the WEP property if its modulus at a point $z$ is bounded from below by a function
of the distance from $z$ to the zero set of $I$.  This is equivalent to a number of
properties, and we establish some consequences of this for $H^\infty/IH^\infty$.

The bulk of the paper is devoted to \emph{wepable} functions, i.e. those inner functions
which can be made WEP after multiplication by a suitable Blaschke product. We prove that
a closed subset $E$ of the unit circle is of finite entropy (i.e. is a Beurling--Carleson
set) if and only if any singular measure supported on $E$ gives rise to a wepable singular
inner function.  As a corollary, we see that singular measures which spread their mass
too evenly cannot give rise to wepable singular inner functions.
Furthermore, we prove that the stronger property of porosity of $E$
is equivalent to a stronger form of wepability (\emph{easy wepability}) for the singular inner functions with support in $E$. 
Finally, we find out the critical decay rate of masses of atomic measures (with no restrictions on support) 
guaranteeing that the corresponding singular inner functions are easily wepable.
\end{abstract}

\section{Introduction}

\subsection{Background.}

Let $H^\infty=H^\infty(\D)$ be the algebra of bounded analytic functions on the unit disc $\mathbb D$ with the norm $\|f \|_\infty = \sup_{z \in \mathbb D} |f(z)|$. A function $I \in H^\infty$ is called
\emph{inner} if it has radial limits of modulus $1$ at almost every point of the unit circle.

Any inner function $I$ factors as $I=BS$ where $B$ is  a Blaschke product and $S$ is a
\emph{singular} inner function, that is, an inner function without zeros in $\mathbb D$. 

A Blaschke product $B$ is called an \emph{interpolating} Blaschke product if its zero set $\Lambda=(z_n)_n$ forms an interpolating sequence for $H^\infty$, that is $H^\infty |\Lambda=\ell^\infty|\Lambda$.
Let $\rho(z,w)$ be the pseudohyperbolic distance between the points $z$ and $w$ in the unit disc $\mathbb{D}$ defined as
$$
\rho(z,w)= \left| \frac{z-w}{1- \overline{w}z} \right|, \qquad  z,w \in \mathbb D.
$$
A celebrated result of Carleson says that this holds if and only if
$\inf_{n \neq m} \rho(z_n , z_m) >0$ and
\begin{equation}
\label{1}
    \sup_{z\in\mathbb D}\sum_n \frac{(1- |z_n|^2) (1- |z|^2)}{|1 - \overline{z_n} z |^2} < \infty.  
\end{equation}
  
It was also proved by Carleson that \eqref{1} is equivalent to the embedding $H^1 \subset L^1 (d \mu)$, where $H^1$ is the standard Hardy space and 
$d \mu = \sum_n (1-|z_n|^2) \delta_{z_n}$, 
$\delta_{z_n}$ being the point mass at $z_n$. In other words, \eqref{1} holds if and only if there exists a constant $C>0$ such that
$
\sum_n (1-|z_n|^2) |f(z_n) | \le C \|f \|_1
$
, for any function $f$ in the Hardy space $H^1$ of the analytic functions in $\mathbb D$ for which
$$
\|f\|_1 = \sup_{0<r<1} \int_0^{2 \pi} |f(r e^{it})| dt < \infty.
$$

It is well known that a Blaschke product $B$ is
an interpolating Blaschke product if and only
if there exists a constant $C=C(B)>0$ such that
\begin{equation}
\label{distzero}
|B(z)| > C \rho(z, Z(B)),
\end{equation}
where $Z(B)$ denotes the zero set of $B$
(see the monographs \cite{Ni1}, \cite[p. 217]{Ni} or \cite{Va}).
This fact easily extends to general inner functions.

\subsection{Weak Embedding Property.}

In 2008, Gorkin, Mortini, and Nikolski \cite{GoMoNi} introduced the following
new class of inner functions. An inner function $I$ satisfies the \emph{Weak Embedding Property (WEP)},
a weaker version of \eqref{distzero},
if for any $\varepsilon >0$ one has
$$
\eta_I(\eps):= \inf \{|I(z)| : \rho(z, Z(I)) > \varepsilon \} >0.
$$

A Blaschke product $I$ with zeros $(z_n)_n$ satisfies the WEP if and only if for every $\varepsilon>0$, 
$$
\sup_{z \in \mathbb D, \,  \inf_n \rho (z, z_n) > \varepsilon} \left\{ \sum \frac{(1-|z_n|^2)(1- |z|^2)}{|1 - \overline{z_n} z|^2} \right\}< \infty \, ,
$$
which is a weakening  of the Carleson embedding property \eqref{1}.

Finite products of interpolating Blaschke products satisfy the WEP
with $\eta_I(\eps) \succeq \eps^N$ and in fact,
a Blaschke product $B$ is the product of $N$ interpolating Blaschke products
if and only if there exists a constant $C=C(B)>0$ such that $|B(z)| > C \rho (z, Z(B))^N$ for any $z \in \mathbb D$ \cite{GoMo}.

However there are other inner functions that satisfy the WEP. In \cite{GoMoNi}, an explicit example
was presented of a Blaschke product satisfying the WEP which cannot factor into a finite product of interpolating Blaschke products. This example was extended and complemented in \cite{NiVa}.
A different class of examples has been  given in \cite{Bo} showing that for every strictly increasing function $\psi : (0,1) \rightarrow (0,1)$ there exists a Blaschke product $B$ satisfying the WEP such that $\eta_B(\eps)  =o ( \psi (\varepsilon))$ as $\varepsilon \to 0$.

\subsection{Operator Theory motivations.}

Given an inner function $I$ consider the quotient algebra $H^\infty / I H^\infty$. The zeros $Z(I)$ of $I$ in $\mathbb D$ are naturally embedded in the maximal ideal space $\mathfrak M$ of $H^\infty / I H^\infty$. It is proved in \cite{GoMoNi} that $I$ satisfies the WEP if and only if $H^\infty / I H^\infty$ has no corona, that is, $Z(I)$ is dense in $\mathfrak M$.

Another condition shown to be equivalent to the WEP in \cite{GoMoNi} is the norm controlled inversion property which says that for any $\varepsilon >0$ , there exists $m(\varepsilon) >0$ such that if $f \in H^\infty $, $\|f\|_{H^\infty} = 1$ and $\inf\{|f(z)|: z \in Z(I) \} > \varepsilon >0$, then $f$ is invertible in $H^\infty / I H^\infty$ and $\|1/f \|_{H^\infty / I H^\infty} 
 \le m (\varepsilon)$.

Consider a vector-valued version of this:
for $f:=(f_1,\dots,f_n) \in (H^\infty)^n$, let
$\|f\|^2_{\infty,n}:= \sup_{z\in \D} \sum_{j=1}^n|f_j(z)|^2$
and for $I$ inner,
$$
\chi_I(f) :=
\inf\{ \|g\|_{\infty,n} : \exists h\in H^\infty :
\sum_{j=1}^n g_j f_j + h I \equiv 1\}.
$$
This is like a ``Corona constant" for the $n$-tuple $f$ in the quotient space
$H^\infty/IH^\infty$.

Following Gorkin, Mortini and Nikolski, for $\delta \in (0,1)$, $n\ge 1$, we define
$$
c_n (\delta, I) := \sup \left\{ \chi_I(f) : \delta^2 \le \inf_{\lambda \in Z(I)} \sum_{j=1}^n |f_j(\lambda)|^2,
\|f\|_{\infty,n} \le 1 \right\},
$$
which is a decreasing function of $\delta$, and
$$
\delta_n(I):= \inf \left\{ \delta : c_n (\delta, I) <\infty \right\}.
$$
It turns out that these values do not depend on $n$.

\begin{proposition}
\label{deltaseq}
For any $n\ge 1$, $\delta_n(I)= \tilde \delta(I) := \inf \{ \eps : \eta_I (\eps) >0 \}$.
\end{proposition}

Another result concerns possible rates of growth of $c_n (\delta, I)$.

\begin{proposition}
\label{cn}
For every decreasing function $\phi : (0,1) \rightarrow (0,\infty)$ there exists a Blaschke product $B$ such that $\delta_n(B)=0$ and 
$c_n (\delta,B)\ge \phi(\delta)$, $0<\delta<1$, $n\ge 1$.
\end{proposition}

Note that by the definition of $\tilde \delta$, the function $I$ satisfies the WEP if and only if $\tilde \delta(I)=0$.

Inner functions $I$ satisfying the WEP can also be described in terms of spectral properties of the model operator acting on the Model Space $K_I = H^2 / I H^2$, see \cite{GoMoNi}.

\subsection{Wepable functions.}

An inner function $I$ is called \emph{wepable} \cite{Bo} if it can enter as a
factor in a WEP inner function, i.e.
 if there exists $J$ inner such that $I J$ satisfies the WEP.
 Clearly, if $I$ is a singular inner function, thus without zeros, it cannot
 be WEP, but it can be wepable.
 It is easy to see that only the Blaschke factor in $J$ will help make $IJ$ a WEP function.

 Let us describe some of the results in \cite{Bo}. Let $dA (z)$ be area measure in the unit disc. An inner function $I$ such that for any $\varepsilon >0$ one has
\begin{equation}
\label{2}
    \int_{\{z : |I(z)| < \varepsilon \} } \frac{dA(z)}{1 - |z|^2} = \infty
\end{equation}
is not wepable. Moreover there exist singular inner functions $I$ satisfying \eqref{2}.
Hence there exists singular inner functions which are not wepable, answering a question in \cite{GoMoNi}.
Condition \eqref{2} is a sort of Blaschke condition and has also appeared  in \cite{LySe}.
It was also shown in \cite{Bo} that condition \eqref{2} does not characterise (non)-wepable inner functions.

\subsection{Results about the support of the singular measure.}

Given a measurable set $E \subset \mathbb T=\partial \mathbb D$, let $|E|$ denote its normalised length, $|\mathbb T|=1$.
Recall that a closed set $E \subset \mathbb T$ with $|E|=0$ has finite entropy (has finite Carleson characteristic, is a Beurling--Carleson set) if
$$
\mathcal E (E) := \sum |J_k| \log |J_k|^{-1} < \infty,
$$
where $(J_k)_k$ are the connected components of $\mathbb T \setminus E$; more precisely, this value is the entropy of the family $(J_k)_k$.
A classical result of Carleson says that a closed set $E \subset \mathbb T$ is the zero set of an analytic function whose derivatives of any order extend continuously to the closed unit disc if and only if $E$ has zero length and finite entropy \cite{Ca}.

Given an inner function $I$ let $ \text{sing} (I)$ denote the set of points of the unit circle where $I$ can not be extended analytically. If $I=BS$ where $B$ is a Blaschke product with zeros 
$(z_n)_n$ and
$$
S_\mu(z)= \exp \left( - \int_0^{2 \pi} \frac{e^{it} + z}{e^{it} -z} d \mu (t) \right), \qquad z \in \mathbb D,
$$
where $\mu$ is a positive singular measure, then $\text{sing} (I) = (\overline { \{z_n \} } \cap \mathbb T) \cup \supp \mu$, where $\supp \mu$ denotes the (closed) support of $\mu$.

\begin{theorem}
\label{thment}
Let $E$ be a closed subset of the unit circle. The following conditions are equivalent:

{\rm (a)} Every singular inner function whose singular set is contained in $E$ is wepable;

{\rm (b)}  $E$ has zero length and finite entropy.
\end{theorem}

The sufficiency of the conditions in (b) is obtained by careful constructions
of Blaschke products which are carried out in Section \ref{proofTh1}.

The necessity of the conditions in (b) is related to estimate \eqref{2} and follows from the following result which may be of independent interest. We identify the unit circle with the interval 
$[0,2\pi)$ and consider the dyadic arcs $[2\pi\cdot k 2^{-n}, 2 \pi\cdot (k+1) 2^{-n})$, $0\le k<2^n$, $n\ge 0$.
Those arcs have normalised length equal to $2^{-n}$.

Given an arc $J \subset \mathbb T$ of center ${\xi}_J$, 
write $z(J) = (1- \frac34|J|) {\xi}_J$. We also denote $Q(J):= \{ r\eit : \eit \in \overline{J}, 1-|J| \le r<1\}$ (the Carleson box associated to $J$)
 and $T(J) := \{ r\eit : \eit \in \overline{J}, 1-|J| \le r\le 1-|J|/2\}$ (the top half of the box).

 Given a finite measure $\mu$ in the unit circle let $P[\mu]$ be its Poisson integral.

\begin{lemma}
\label{L1}
Let $E$ be a closed subset of the unit circle. The following conditions are equivalent:

{\rm (a)}  $E$ has zero length and finite entropy;

{\rm (b)}  $\sum |J| < \infty$, where the sum is taken over all dyadic arcs $J$ such that $J \cap E \neq \varnothing$;

{\rm (c)}  For any positive singular measure $\mu$ whose support is contained in  $E$ and any $C>0$ one has $\sum |J| < \infty$, where the sum is taken over all dyadic arcs 
$J \subset\mathbb T$ such that $P[\mu] (z(J)) \ge C$.
\end{lemma}

The condition in (c) can be understood as a discrete version of \eqref{2} with $I= S_\mu$.

The condition in (b) can be seen as a discrete version of $\int_{\Gamma(E)} \frac{dA(z)}{1 - |z|^2}  < \infty$,
where $\Gamma(E)$ denotes the union of all the Stolz angles with vertex on a point of $E$. For
a related result, see \cite[Lemma A.1]{EFKeRa}.

\subsection{Results about regularity of singular measures.}

Positive singular measures can fairly distribute their mass. For instance, there exist singular probability measures $\mu$ on the unit circle such that
$$
\sup \Bigl\{\Bigl|\frac{\mu(J)}{\mu (J')} - 1\Bigr| + \Bigl|\frac{\mu(J)-\mu(J')}{|J|} \Bigr|\Bigr\} \to 0 \quad \text{ as } |J| \to 0,
$$
where the supremum is taken over any pair of adjacent arcs $J, J' \subset \mathbb T$ of the same length (see \cite{AlAnNi}). As a consequence of Theorem~\ref{thment} we will prove that positive singular measures $\mu$ such that $S_\mu$ is wepable cannot distribute their mass as evenly. Actually a Dini type condition governs the growth of the density of such measures.

\begin{corollary}
\label{regular}
{\rm (a)} Let $\mu$ be a positive singular measure on the unit circle and consider $w(t) = \sup \mu(J)$, where the supremum is taken over all arcs $J \subset \mathbb T$ with $|J|= t$. Assume that 
$$
\sum_{n\ge 1} \frac{2^{-n}}{w(2^{-n})} = \infty .
$$
Then $S_\mu$ satisfies condition \eqref{2} and, hence, it is not wepable.

{\rm (b)} Let $w: [0,1] \rightarrow [0, \infty)$ be a nondecreasing function with $w(0)=0$ such that $w(2t) < 2 w(t)$ for any $t>0$. Assume that 
$$
\sum_{n\ge 1} \frac{2^{-n}}{w(2^{-n})} < \infty .
$$
Then there exists a positive singular measure $\mu$ in the unit circle satisfying $\mu (J) < w(|J|)$ for any arc $J \subset\mathbb T$,  
such that its support has zero length and finite entropy, and hence $S_\mu$ is wepable.
\end{corollary}

\begin{definition}
A closed subset $E$ of the unit circle is called \emph{porous}
if there exists a constant $C>0$ such that for any arc $J \subset\mathbb T$,
there exists a subarc $J' \subset J \setminus E$ with $|J'| > C |J|$.
\end{definition}

The porosity condition (the Kotochigov condition, the (K) condition) appears naturally in the free interpolation problems for different classes of analytic functions smooth up to the boundary, 
see \cite{dyn}.

The next auxiliary result is a scale invariant version of Lemma \ref{L1}.

\begin{lemma}
\label{lempor}
Let $E$ be a closed subset of the unit circle. The following conditions are equivalent:

{\rm (a)} $E$ is porous;

{\rm (b)} There exists a constant $C>0$ such that for any dyadic arc $J$ one has
 $$
 \sum_{I \in E(J)} |I| \le C |J|,
 $$
 where $E(J)$ is the family of the dyadic arcs $I \subset J$ such that $I \cap E \neq \varnothing$;

{\rm (c)} There exists a constant $C>0$ such that for any finite positive measure $\mu$ with support contained in $E$, any number $A>0$ and any dyadic arc $J \subset\mathbb T$, one has
$$
\sum_{I \in E(A)} \frac{|I||J|}{|1- \overline{z(I)} z(J)|^2} \le \frac CA P[\mu] (z(J)), 
$$
where $E(A)$ is the family of the dyadic arcs $I$  such that $P[\mu] (z(I)) \ge A$.
\end{lemma}

Let $S$ be a wepable inner function. It may happen that any Blaschke product $B$ such that $BS$ satisfies the WEP must have some of its zeros $(z_n)_n$ located at points where $|S|$ is close to $1$; more precisely,  $\limsup_{n \to \infty} |S(z_n)| = 1$. (Theorems~\ref{thment} and \ref{thmpor} together prove the existence of such $S$). 

\begin{definition}
An inner function $S$ will be called \emph{easily wepable} if there exists a constant $m<1$ and a Blaschke product $B$ such that $SB$ satisfies the WEP and $Z(B) \subset \{ z \in \mathbb D : |S(z)|<m \}$.
\end{definition}

\begin{theorem}
\label{thmpor}
Let $E$ be a closed subset of the unit circle. The following conditions are equivalent:

{\rm (a)} Every singular inner function whose singular set is contained in $E$ is easily wepable;

{\rm (b)} The set $E$ is porous.
\end{theorem}

In \cite{MoNi} it was proved that a closed set $E$ of the unit circle is porous if and only if for any singular inner function $S$ whose singular set is contained in $E$ and any $a \in \mathbb D \setminus \{0 \} $ the inner function $(S -a)/ (1- \overline{a} S)$ is a finite product of interpolating Blaschke products. 

Now we describe the critical decay rate of masses of atomic measures (with no restrictions on support) 
guaranteeing that the corresponding singular inner functions are easily wepable. 

\begin{theorem}
\label{thmatom} 
Let $(b_s)_{s\ge 1}$ be a non-increasing summable sequence of positive numbers. The following conditions are equivalent:

{\rm (a)} Every atomic singular inner function with point masses $(b_s)_{s\ge 1}$ is easily wepable;

{\rm (b)} 
\begin{equation}
\label{kl5}
b_s\asymp \sum_{k\ge s}b_k, \qquad s\ge 1.
\end{equation}
\end{theorem}

Note that given any decreasing sequence of masses, they can give rise to a measure $\mu$ with easily
wepable singular function $S_\mu$, simply by locating the masses at points of the form $\exp (i2^{-n})$,
for instance, and applying Theorem \ref{thmpor}.
It would be interesting to have a similar statement to Theorem \ref{thmatom} with ``wepable" instead of
``easily wepable"; in particular to know whether any condition weaker than \eqref{kl5} can imply automatic wepability, and what rate of decrease of the $(b_s)$ guarantees that there always
exists some choice of location of the point masses with produces a non-wepable $S_\mu$. In particular, 
the construction in the proof of \cite[Proposition 6]{Bo} shows that there exists a non-wepable atomic singular inner function as soon as the point masses decay no more rapidly than 
$1/(n(\log n)^2)$, $n\to\infty$.

\subsection{Organization of the paper.}
In Section~\ref{proofLm1},
we prove Lemmas~\ref{L1} and \ref{lempor},
and therefore the necessity part of Theorem~\ref{thment}.  In Section~\ref{proofTh1},
we prove the remaining part of Theorem~\ref{thment} and Corollary~\ref{regular}.  In Section~\ref{new},
we give the proof of Theorem~\ref{thmatom}.  In Section~\ref{proofpor},
we give the proof of Theorem~\ref{thmpor}, which deals with a situation where
the entropy of the singular set is very well controlled.  Finally, the proofs of 
Propositions~\ref{deltaseq} and \ref{cn}, which are quite independent from the rest, appear in Section~\ref{proofcsts}.

The letter $C$ will denote a constant whose value may change from line to line.

We denote by $\mathcal D=\cup_{n\ge 0} \mathcal D_n$ the family of the dyadic arcs, with
$\mathcal D_n=\{J\subset \mathcal D:|J|=2^{-n}\}$. Note that $\card \mathcal D_n=2^n$.

Given an arc $J \subset\mathbb T$ of center $\xi$ and length $|J|$ and $M>0$ let $MJ$ be the arc of the unit circle of center $\xi$ and length $M|J|$.

Let $z,w\in\mathbb D$. Later on, we use the following standard estimates:
\begin{gather}
\label{kl8} \frac12\cdot\frac{(1-|z|^2)(1-|w|^2)}{|1-\bar w z|^2}\le \log\Bigl|\frac{1-\bar w z}{z-w}\Bigr|,\\
\label{kl9} \log\Bigl|\frac{1-\bar w z}{z-w}\Bigr| \le C(\delta)\frac{(1-|z|^2)(1-|w|^2)}{|1-\bar w z|^2},\qquad \rho(z,w)\ge \delta>0.
\end{gather}

\subsection*{Acknowledgments}
We are grateful  to Nikolai Nikolski for stimulating discussions. 

This work was initiated in 2011 when the third author was invited by the Centre de Recerca Matem\`atica in the framework of the thematic semester on Complex Analysis and Spectral Problems.

\section{Proofs of Lemmas~\ref{L1} and \ref{lempor}}
\label{proofLm1}

\begin{proof}[Proof of Lemma~\ref{L1}]
(a)$\Leftrightarrow$(b) Let $\mathbb T\setminus E$ be the disjoint union of the arcs $I_k$, $k\ge 1$. Suppose first that $|E|=0$.
Then
\begin{multline*}
\sum_{J\in \mathcal D,\,J\cap E\not=\varnothing}|J|=
\sum_{k\ge 1}\sum_{J\in \mathcal D,\,J\cap E\not=\varnothing,\,J\cap I_k\not=\varnothing}|J\cap I_k|\\ =
\sum_{k\ge 1}\sum_{n\ge 0}\sum_{J\in \mathcal D_n,\,J\cap E\not=\varnothing,\,J\cap I_k\not=\varnothing}|J\cap I_k|\\ \asymp
\sum_{k\ge 1}\Bigl(\sum_{0\le n\le \log(1/|I_k|)}|I_k|+\sum_{n>\log(1/|I_k|)}2^{-n}\Bigr)\asymp \mathcal E(E).
\end{multline*}
Next, if $|E|>0$, then
$$
\sum_{J\in \mathcal D,\,J\cap E\not=\varnothing}|J|=\infty.
$$

(b)$\Rightarrow$(c) Arguing as above,
\begin{multline*}
\sum_{J\in \mathcal D,\,P[\mu](z(J))\ge C}|J| = \sum_{k\ge 1}\sum_{n\ge 0}\sum_{J\in \mathcal D_n,\,P[\mu](z(J))\ge C,\,J\cap I_k\not=\varnothing}|J\cap I_k|  \\
\lesssim \sum_{k\ge 1}\Bigl(\sum_{0\le n\le 2\log(1/|I_k|)}|I_k|+\sum_{n>2\log(1/|I_k|)}2^{-n/2}\Bigr)\asymp \mathcal E(E).
\end{multline*}

(c)$\Rightarrow$(a) If $|E|>0$, then we can just take $\mu_0=K\chi_E \,dm$ with $K$ to be chosen later on. By the Lebesgue density theorem, for a subset
$E_1$ of $E$, $|E_1|\ge |E|/2$ and for some $\delta>0$ we have
$$
\frac{|E\cap J|}{|J|}\ge \frac12
$$
for every arc $J$ such that $J\cap E_1\not=\varnothing$, $|J|\le \delta$. Hence, for $K\ge K(C)$ we obtain
$$
\sum_{J\in\mathcal D,\, P[\mu_0](z(J))\ge C}|J|\ge \sum_{J\in\mathcal D,\, |J|\le\delta,\,J\cap E_1\not=\varnothing}|J|=\infty.
$$
Next, we can replace $\mu_0$ by a Cantor type singular measure $\mu_1$ while keeping the sum
$$
\sum_{J\in\mathcal D,\, P[\mu_1](z(J))\ge C}|J|
$$
infinite.

Now, suppose that (c) holds and $|E|=0$, $\mathcal E(E)=\infty$, so that (a) and, hence, (b) do not hold. 
Let $\mathbb T\setminus E$ be the disjoint union of the arcs $I_k=(a_k,b_k)$, $k\ge 1$,
and take
$$
\mu=K\sum_{k\ge 1}|I_k|(\delta_{a_k}+\delta_{b_k})
$$
with $K$ to be chosen later on. Given $J\in\mathcal D$, if  $J\cap E\not=\varnothing$, then $\mu(J)\ge K|J|$, and for $K\ge K(C)$ we have
$P[\mu](z(J))\ge C$. Therefore,
$$
\sum_{J\in\mathcal D,\, P[\mu](z(J))\ge C}|J|\ge \sum_{J\in\mathcal D,\, J\cap E\not=\varnothing}|J|=\infty. 
$$
\end{proof}

\begin{proof}[Proof of Lemma~\ref{lempor}]
(a)$\Rightarrow$(b) If $E$ is porous, then there exists $a\in\mathbb N$ such that for every $n\ge 0$, $J\in\mathcal D_n$, and for every $m\ge 0$, the set $J\cap E$ is covered by 
$2^{m-s}$ arcs $I\in \mathcal D_{n+m}$, $sa\le m<(s+1)a$. Fix $J\in \mathcal D$, $|J|=2^{-n}$. Then 
\begin{gather*}
\sum_{I\in\mathcal D,\,I\subset J,\, I\cap E\not=\varnothing} |I|\le \sum_{m\ge n}\sum_{I\in\mathcal D_m,\,I\subset J,\, I\cap E\not=\varnothing} |I| \\  \le 
\sum_{s\ge 0}\sum_{sa\le m<(s+1)a} 2^{m-s}2^{-n-m}=\sum_{s\ge 0}a2^{-n-s}=2a|J|.
\end{gather*}

(b)$\Rightarrow$(a) If $E$ is not porous, then for every $N\ge 1$ there exist $n\ge 0$, $J\in\mathcal D_n$ such that if $I\in \mathcal D_{n+N}$, $I\subset J$, then $I\cap E\not=\varnothing$. 
Then
\begin{gather*}
\sum_{I\in\mathcal D,\,I\subset J,\, I\cap E\not=\varnothing} |I|\ge \sum_{n\le s\le n+N}\sum_{I\in\mathcal D_s,\,I\subset J,\, I\cap E\not=\varnothing} |I|\\ =
\sum_{n\le s\le n+N} |J|=(N+1)|J|.
\end{gather*}

(c)$\Rightarrow$(a) As above, if $E$ is not porous, then for every $N\ge 1$ we can find $J=[2\pi \cdot k2^{-n},2\pi \cdot(k+1)2^{-n})\in\mathcal D_n$ 
and points $x_s\in E\cap [2\pi \cdot(k2^N+s)2^{-n-N},2\pi \cdot(k2^N+s+1)2^{-n-N})$, $0\le s<2^N$.
Set
$$
\mu=10\cdot 2^{-n-N}A\sum_{0\le s<2^N}\delta_{x_s}.
$$
Then
$$
P[\mu](z(I))\ge A, \qquad I\in\mathcal D_m,\, n\le m\le n+N,\, I\subset J,
$$
and
$$
\sum_{I\in\mathcal D,\, P[\mu](z(I))\ge A} \frac{|I||J|}{|1- \overline{z(I)} z(J)|^2} \gtrsim 
\sum_{n\le m\le n+N} 2^{m-n}\cdot\frac{2^{-m}2^{-n}}{2^{-2n}}=N+1.
$$
For large $N$ this contradicts to (c), because 
$$
P[\mu](z(J))\le CA.
$$

To complete the proof of our lemma, we need an auxiliary statement. 

\begin{lemma}
\label{kl7}
Let $u$ be a function positive and harmonic on the unit disc, let $A>0$, and let $\mathcal G$ be a subfamily of $\mathcal D$ such that 
the arclength $ds$ on $L=\cup_{J\in\mathcal G}\partial T(J)$ is a Carleson measure,
$$
\sup_{z\in\mathbb D}\sum_{J\in\mathcal G}\int_{\partial T(J)}\frac{1-|z|^2}{|1-\bar w z|^2}\,ds(w)\le B.
$$ 
Assume that $u\ge A$ on $L$. Then for every $J\subset \mathcal D$ we have 
$$
\sum_{I\in\mathcal G}\frac{|I||J|}{|1- \overline{z(I)} z(J)|^2}\le \frac{CB}{A}u(z(J))
$$
for some absolute constant $C$.
\end{lemma}

\begin{proof} Consider the function 
$$
h(z)=\sum_{J\in\mathcal G}\int_{\partial T(J)}\log\Bigl|\frac{1-\bar w z}{z-w}\Bigr|\,\frac{ds(w)}{1-|w|}.
$$
It is harmonic on $\mathbb D\setminus L$ and by \eqref{kl9},  
$$
CBu(z)\ge Ah(z) 
$$
for $z\in \mathbb T\cup L$. By the maximum principle, 
$$
CBu(z(J))\ge Ah(z(J)),
$$
and hence, by \eqref{kl8}, 
$$
C_1Bu(z(J))\ge A\sum_{I\in\mathcal G}\frac{|I||J|}{|1- \overline{z(I)} z(J)|^2}.
$$
\end{proof}

It remains to prove the implication (a)$\Rightarrow$(c). Set $u=P[\mu]$. Arguing as in the part 
(a)$\Rightarrow$(b) we obtain that 
$$
\sum_{I\in \mathcal D,\,I\subset J,\,3I\cap E\not=\varnothing}|I|\le C|J|,\qquad  J\subset \mathcal D.
$$
Hence, the arclength $ds$ on $\cup_{I\in \mathcal D,\,3I\cap E\not=\varnothing}\partial T(I)$ is a Carleson measure.
Fix $A>0$ and denote by $\mathcal G$ the family of all $I\in\mathcal D$ such that $3I\cap E\not=\varnothing$ and $u(z(I))\ge A$.
Fix $J\in\mathcal D$. Applying Lemma~\ref{kl7}, we obtain that 
$$
u(z(J))\ge CA\sum_{I\in\mathcal G}\frac{|I||J|}{|1- \overline{z(I)} z(J)|^2}.
$$
Now we need only to estimate 
$$
\sum_{I\in \mathcal D,\,3I\cap E=\varnothing,\,u(z(I))\ge A}\frac{|I||J|}{|1- \overline{z(I)} z(J)|^2}.
$$

We set
$$
\mathcal H=\{I\in\mathcal D:3I\cap E=\varnothing,\,u(z(I))\ge A\}
$$
and 
$$
\mathcal H_k=\{I\in\mathcal H:2^{k-1}A\le u(z(I))< 2^kA\},\qquad k\ge 1,
$$
so that $\mathcal H=\cup_{k\ge 1}\mathcal H_k$.
If $I\in\mathcal H$, then an easy estimate of the Poisson integral shows that 
$$
\frac{u(z(L))}{u(z(I))}\asymp \frac{|L|}{|I|},\qquad L\in\mathcal D,\, L\subset I.
$$
Hence, every arc $I\in\mathcal H_k$, $k\ge 1$, contains at most $C$ subarcs $I'\in\mathcal H_k$. Therefore, 
the arclength $ds$ on $\cup_{I\in\mathcal H_k}\partial T(I)$ is a Carleson measure with Carleson constant uniformly bounded in $k\ge 1$. 
Lemma~\ref{kl7} gives now that 
$$
2^{-k}u(z(J))\ge CA\sum_{I\in\mathcal H_k}\frac{|I||J|}{|1- \overline{z(I)} z(J)|^2}.
$$
Summing up in $k\ge 1$ we complete the proof.
\end{proof}

\section{Proof of Theorem 1 and Corollary 1}
\label{proofTh1}

The proof of the sufficiency of Theorem 1 uses the following auxiliary result.

\begin{lemma}
\label{intjn}
Let $E$ be a closed subset of the unit circle of zero length and finite entropy.
Let ${\mathcal G} = \{J_n \}$ be the family of maximal dyadic subarcs $J_n \subset\mathbb T$
such that $2 J_n \subset\mathbb T\setminus E$. Then

{\rm (a)} The interiors of $J_n$ are pairwise disjoint and $\mathbb T\setminus E = \cup J_n$.

{\rm (b)} We have $\sum_n |J_n| \log |J_n|^{-1} < \infty$.

{\rm (c)} If $J_n$ and $J_m$ are in the same connected component of $\mathbb T\setminus E $ and $\overline{J_n} \cap \overline{J_m} \neq \varnothing$, then $4 |J_n| \ge |J_m| \ge |J_n| /4$.

{\rm (d)} Let $J$ be a connected component of $\mathbb T\setminus E$ and consider the subfamily ${\mathcal G}(J) = \{L_n \}$ of the arcs $L \in {\mathcal G}$ with $L \subset J $ ordered so that $|L_{n+1}| \le |L_n|$ for any $n$. Then $|L_1| \ge |J|/8$, $|L_k| / 4 \le |L_{k+1}| \le |L_k|$ and $|L_{k+4}| \le |L_k|/2$ for any $k\ge 1$.
\end{lemma}

\begin{proof}
The maximality gives that the interiors of $J_n$ are pairwise disjoint. It is also clear that $\mathbb T\setminus E = \cup J_n$. 
Let us now prove (c). Assume that $|J_n| \ge |J_m|$. Since $2 J_n \subset\mathbb T\setminus E $, every dyadic arc $J$ adjacent to $J_n$ of length $|J|=|J_n|/4$ satisfies 
$2J \subset \mathbb T \setminus E $. 
Since $\overline{J_n} \cap \overline{J_m} \neq \varnothing$ we can assume that $J \cap J_m \neq \varnothing$. By maximality, $J \subset J_m$. Hence $|J_m| \ge |J|= |J_n|/4$ and (c) is proved. 
Let us now prove (d). The maximality gives that $|L_1| \ge |J|/8$ and (c) gives that $|L_k|/4 \le |L_{k+1}|$. 
Assume that $|L_{k+4}| > |L_k|/2$. Then $|L_{k+i}|=|L_k|$, $0\le i\le 4$. Then there would exist $i\in\{0,1,2,3,4\}$ such that the dyadic arc $L$ containing $L_{k+i}$ of length $|L|=2|L_{k+i}|$ satisfies $2L \subset\mathbb T \setminus E $. 
This would contradict the maximality of $L_{k+i}$ and (d) is proved. Finally, let us prove (b). Take a connected component $J$ of $\mathbb T\setminus E$. The estimates in (d) give that
$$
\sum_{L \in {\mathcal G}(J)} |L| \log |L|^{-1} \le C |J| \log |J|^{-1},
$$
and (b) follows because $E$ has finite entropy.
\end{proof}

\begin{proof}[Proof of Theorem \ref{thment}]
The necessity part follows from Lemma \ref{L1} and property \eqref{2}.

Let us now prove the sufficiency part. Let $\mu$ be a positive singular measure whose support $E$ has zero length and finite entropy. Let ${\mathcal G}$ be the family given by Lemma
\ref{intjn} of maximal dyadic arcs of the unit circle whose double is contained in
$\mathbb T\setminus E$. Let $\mathcal F$ be the family of dyadic arcs of
$\mathbb T$ which are not contained in any arc of ${\mathcal G}$. Then
\begin{equation}
\sum_{J \in \mathcal F} |J| < \infty.
\label{e4}
\end{equation}
Indeed, for every $J \in \mathcal F$, $2J$ is the union of four dyadic intervals of the same length such that at least one of them intersects $E$. 
Since $E$ has zero length and finite entropy, Lemma~\ref{L1} gives that
$$
\sum_{J \in \mathcal F} |J|\le \sum_{J \in \mathcal F} |2J|\lesssim \sum_{J \in \mathcal D,\, J \cap E \ne \varnothing} |J| < \infty.
$$
We want to find a Blaschke product $B$ such that $SB$ satisfies the WEP.
We will describe now the first family of zeros of $B$. By \eqref{e4},
one can pick a sequence of integers $k_j$ increasing to infinity such that
$$
\sum_{j=1}^{\infty}  k_j \sum_{J \in \mathcal F\cap\mathcal D_j} |J| < \infty.
$$
Next, for each $J\in\mathcal F\cap\mathcal D_j$, we choose a set $\Lambda(J)=\{z_n (J) : 1\le n\le k_j \}$ uniformly distributed in $T(J)$.  
Uniform distribution means that 
$$
\min_{z,w\in \Lambda(J),\, z\not=w}|z-w|\asymp \max_{z\in T(J)}\dist(z,\Lambda(J))\asymp \min_{z\in\Lambda(J)}\dist(z,\partial T(J)).
$$
Let $B_1$ be the Blaschke product with zeros $Z(B_1) =\cup_{J\in\mathcal F} \Lambda(J)$.
Since $k_j$ tends to infinity, for any $0< \delta <1$ there exists $r=r(\delta) < 1$ such that
\begin{equation}
\{z \in \mathbb D : \rho (z, Z(B_1)) > \delta \} \subset \{z \in \mathbb D : |z|<r \} \bigcup {\cup}_{J \in {\mathcal G}} Q(J).
\label{e5}
\end{equation}
Next we will construct certain additional zeros of $B$ which are contained in $\cup_{J\in\mathcal G}Q(J)$. Set $f=S B_1$. Since $\inf_{J\in\mathcal G} \rho(z(J), Z(B_1)) >0$,
we have $\log |f(z(J))|^{-1} \lesssim (1- |z(J)|)^{-1}$, $J\in\mathcal G$.
Now, applying Lemma~\ref{intjn}~(b), we obtain 
$$
\sum_{J\in\mathcal G} |J| \log \log |f(z(J))|^{-1} < \infty.
$$
Next, for every $J\in\mathcal G$ and every $z \in Q(J)$ with $1-|z|<|J|/8$, we have 
\begin{equation}
\log|f(z)|^{-1} \lesssim \frac{1-|z|}{1-|z(J)|} \log |f(z(J))|^{-1}.
\label{e7}
\end{equation}
To prove \eqref{e7} consider the measure $d \sigma = d \mu + \sum_{z\in Z(B_1)} (1-|z|^2) \delta_z$, 
where 
$d\mu$ is the positive singular measure
associated to $S$. Fix for a moment $J \in \mathcal G$. There exists a constant $C>0$
such that if $z \in Q(J)$ and $0< 1-|z| < |J|/ 8$, then $\rho(z, Z(B_1))\ge C$.
Then, by \eqref{kl9}, there exists a constant $C_1 >0$ such that
$$
\log |f(z)|^{-1} \le C_1 \int_{\overline{\mathbb D}} \frac{1-|z|^2}{|1- \overline{w} z|^2} d \sigma (w).
$$
Applying Lemma~\ref{intjn}~(c) and using that the support of $\mu$ does not intersect $2J$,  
for every $w$ in the support of $\sigma$ and every $z \in Q(J)$ we obtain that $|1- \overline{w} z| \ge C_2 |1 - \overline{w} z(J) |$, where $C_2>0$ is a constant. Hence,
$$
\log |f(z)|^{-1} < \frac{C_1 (1-|z|^2)}{C_2^2 (1-|z(J)|^2)}  \int_{\overline{\mathbb D}} \frac{1-|z(J)|^2}{|1- \overline{w} z(J)|^2} d \sigma (w)  .
$$
By \eqref{kl8}, this integral is bounded by a fixed multiple of  $\log|f(z(J))|^{-1}$, and estimate \eqref{e7} follows.

In particular, for any constant $K>0$ there exists $K_1 = K_1 (K)>0$ such that for every
$J\in\mathcal G$ we have 
\begin{equation}
|f(z)| > K_1 \quad \text{if} \quad z \in Q(J) ,\, 1-|z| < \frac{K|J|}{\log|f(z(J))|^{-1}}.
\label{e8}
\end{equation}
Let $M(J)=\log_2 |J|^{-1}$. Pick a sequence of positive integers $t_k$ increasing to infinity and such that $t_{k+1}\le t_k+1$ and 
$$
\sum_{J\in\mathcal G} |J| \sum_{k=1}^{M(J)} t_k < \infty.
$$
For $1\le k\le M(J)$ consider the strips
$$
\Omega_k = \Omega_k(J) = \{z \in Q(J) : 2^{k-1} |J|^2 < 1- |z| \le 2^{k} |J|^2 \}
$$
and choose sets $\Lambda_k(J)=\{z_j (\Omega_k)\}_{1\le j\le s_k}$ of $s_k= t_k 2^{M(J) -k}$ points uniformly distributed in $\Omega_k$. Let $\Lambda (J) = \cup_{1\le k\le M(J)}\Lambda_k(J)$. 
Then
$$
\sum_{z\in\Lambda(J)}(1-|z|)=\sum_{k=1}^{M(J)} \sum_{j=1}^{s_k} (1-|z_j (\Omega_k)|) \le |J| \sum_{k=1}^{M(J)} t_k.
$$
Thus, the set $\cup_{J\in\mathcal G}\Lambda(J)$ satisfies the Blaschke condition. Let $B_2$ be the Blaschke product with these zeros.

We will now show that $S B_1 B_2$ satisfies the WEP.
Fix $\eta >0$ and take $z \in \mathbb D$ such that $\rho (z , Z(B_1 B_2)) > \eta$.
Applying \eqref{e5} we obtain that either $|z| \le r(\eta)<1$ or $z \in Q(J)$ for some $J\in\mathcal G$.
In the first case $|S(z) B_1 (z) B_2 (z)| > \varepsilon (\eta) >0$.
Assume that $z \in Q(J)$  for some fixed $J\in\mathcal G$. Since $\lim_{k\to\infty} t_k = \infty$ and $\rho (z , Z(B_2)) > \eta$,
there exists $C_1 = C_1 (\eta) >0$ such that $1-|z| < C_1 |J|^2 $.
Furthermore, $z\in\cup_{0\le k\le k(\eta)}\Omega_k(J)$, where 
$$
\Omega_0(J) = \{z \in Q(J) : 1-|z| \le |J|^2 \}.
$$
Since $\log|f(z(J))|^{-1} < C/|J|$, applying \eqref{e8} we obtain that $|f(z)|\ge C_2$ for some $C_2=C_2(\eta)>0$. 
Let $J_-$ and $J_+$ be two arcs of the family $\mathcal G$ contiguous to $J$. Factor $B_2 = B_3 B_4$, where $B_3$ is the Blaschke product with zeros 
$\Lambda (J_-) \cup \Lambda (J) \cup \Lambda (J_+)$. By Lemma~\ref{intjn}~(c), $\dist(Z(B_4),Q(J))\gtrsim |J|$. Hence there exists a constant $C>0$ such that $|1- \overline{w}z| >C |J|$ for any zero $w$ of $B_4$. Therefore, by \eqref{kl9}, 
$$
\log |B_4(z)|^{-1} \le C_3 \sum_{w\in Z(B_4)} \frac{(1-|z|^2)(1-|w|^2)}{|1 - \overline{w}z|^2} \le C_4 \sum_{w\in Z(B_4)} (1-|w|^2).
$$
On the other hand, if we reorder the zeros of $B_3$ as $Z(B_3)=(z_s)_s$ and use that $\rho(z, Z(B_2)) > \eta$,
estimate \eqref{kl9} gives that there exists a constant $C_5 = C_5 (\eta ) >0$ such that
\begin{multline*}
\log|B_3 (z)|^{-1} \\ \le  C_5 \frac{\sum_{z_s\in Q(z)} (1-|z_s|)}{1-|z|} + C_5 \sum_{j=1}^{\infty} \frac{1}{2^{2j} (1-|z|) }\sum_{z_s\in 2^j Q(z) } (1-|z_s|), 
\end{multline*}
where $Q(z)=Q(\tilde J)$ is the Carleson box defined by the arc $\tilde J$ satisfying $z=z(\tilde J)$, $2^j Q(z)=Q(2^j \tilde J)$. 

Next, again by Lemma~\ref{intjn}~(c), $\log_2|J|^{-1}\asymp \log_2|J_\pm|^{-1}$. Hence, the densities of zeros $z_s\in Q(z)$ are bounded by $t_{k(\eta)}$, 
$$
\frac1{1-|z|}\sum_{z_s\in Q(z)} (1-|z_s|)\le \sum_{1\le k\le k(\eta)}t_k\le C_6(\eta).
$$
Similarly, 
\begin{multline*}
\frac1{1-|z|}\sum_{z_s\in 2^jQ(z)} (1-|z_s|)\le 2^j\sum_{1\le k\le k(\eta)+j}t_k \\ \le 2^j(C_6(\eta)+jt_{k(\eta)}+j(j+1)/2), 
\end{multline*}
and
$$
\sum_{j=1}^{\infty} \frac{1}{2^{2j} (1-|z|) }\sum_{z_s\in 2^j Q(z) } (1-|z_s|) \le C_7(\eta).
$$
\end{proof}

\begin{proof}[Proof of Corollary~\ref{regular}]
(a) Fix $0 < \varepsilon < 1$ and consider $A(\varepsilon) = \{ z \in \mathbb D : |S_\mu(z)| < \varepsilon \}$. 
Let $\mathcal L_n$ be the collection of dyadic arcs $J\in\mathcal D_n$ such that $\mu (J) > 10 |J| \log {\varepsilon}^{-1}$ and let $\mathcal L=\cup_{n\ge 0}\mathcal L_n$. 
Since $\log|S_\mu(z)|^{-1}>\mu(J)/(10|J|)$ for $z\in T(J)$, we deduce that 
$$
 \int_{A(\varepsilon)} \frac{dA(z)}{1-|z|^2} \ge C \sum_{J\in\mathcal L} |J| =  C \sum_n 2^{-n} a_n,
$$
where $a_n$ is the number of dyadic arcs in the collection $\mathcal L_n$. Since $ \mu$ is singular, there exists $n_0 = n_0 (\varepsilon) >0$ such that for $n \ge n_0$, one has
$$
\sum_{J\in\mathcal L_n} \mu (J) \ge \frac{1}{2} \mu(\mathbb T).
$$
We deduce that $a_n w(2^{-n}) \ge \mu (\mathbb T) / 2$. Thus, 
$$
\int_{A(\varepsilon)} \frac{dA(z)}{1-|z|^2}  \ge C \mu (\mathbb T) \sum_{n \ge n_0} \frac{2^{-n}}{w(2^{-n})},
$$
which finishes the proof of (a).

(b) We may assume that $w(1)=1$. Set $n_1 =0$ and for $k=2,3,\ldots$ let $n_k$ be the smallest positive integer such that $w(2^{-n_k}) < 2^{-k +1}$. Since $2^{-k+1} \le w(2^{-n_k + 1}) < 2 w(2^{-n_k})$, we have $2^{-k} < w(2^{- n_k}) < 2^{-k +1}$. Hence for any $k=1,2 \ldots$ and any integer $n$ with $n_k \le n < n_{k+1}$ we have 
$2^{-k}\le w(2^{-n})< 2^{-k +1}$.
Let 
$\mathcal L_k=\mathcal D_{n_k}$,  
$k\ge 1$.
The measure $\mu $ will be defined by prescribing inductively its mass $\mu (J)$ over any dyadic arc $J \in \cup_k \mathcal L_k$. Define $\mu(\mathbb T)=1$.
Assume that $\mu (J)$ has been defined for any arc $J\in\mathcal L_k$ and we will define the mass of $\mu$ over dyadic arcs of $\mathcal L_{k+1}$. Fix $J\in\mathcal L_k$ and let 
$\mathcal G(J)$ be the family of arcs in $\mathcal L_{k+1}$ contained in $J$. If $\mu (J)=0$, define $\mu (L) =0 $ for any dyadic arc $L \in \mathcal G(J)$. If $\mu (J) >0$, pick two (arbitrary) arcs $J_i \in \mathcal G(J)$, $i=1,2$, define $\mu (J_i) = \mu (J) / 2$ for $i=1,2$  and $\mu (L) =0 $ for any other $L \in \mathcal G(J)$ with $L \neq J_i$, $i=1,2$. In other words, in each dyadic arc of generation $n_k$ of positive measure $\mu$, this measure distributes its mass among two (arbitrary) arcs of generation $n_{k+1}$ and gives no mass to the others.
Let $J$ be a dyadic arc of the unit circle with $2^{- n_{k+1}} < |J| \le 2^{- n_k}$ for a certain integer $k=k(J)$. By construction, $\mu(J) \le 2^{-k} \le w(|J|)$. Since any arc is contained in 
the union of two dyadic arcs of comparable length, there exists a constant $C>0$ such that $\mu (J) < C w (|J|)$ for any arc $J \subset\mathbb T$. 
Next we will show that the support of $\mu$  has finite entropy. 
For any integer $k\ge 1$ there are $2^{k-1}$ arcs $J$ of $\mathcal L_k$ with $\mu (J) >0$. Similarly if $n_k< n\le n_{k+1}$, then there are at most $2^k$ dyadic arcs $J$ of normalised length 
$2^{-n}$ with $\mu(J)>0$. Then
\begin{multline*}
\sum_{J \in \mathcal D,\, \mu(J) >0} |J| \le  1+\sum_{k=1}^\infty \sum_{n=n_k+1}^{n_{k+1}} 2^{-n} 2^{k} \\ \le 1+ \sum_{k=1}^\infty 2^k 2^{-n_k} \le  1+2 \sum_{k=1}^\infty \frac{ 2^{-n_k}}{w(2^{-n_k})}\le  1+2 \sum_{n\ge 0} \frac{ 2^{-n}}{w(2^{-n})}.
\end{multline*}
Applying Lemma~\ref{L1} we deduce that the support of $\mu$ has finite entropy  and zero Lebesgue measure. By Theorem~\ref{thment}, $S_\mu$ is wepable.
\end{proof}

\begin{remark} Let $\mu$ be a positive singular measure in the Zygmund class, that is, there exists a constant $C>0$ such that for any pair of contiguous arcs $J, J' \subset\mathbb T$ of the same length one has
$$
|\mu (J)-\mu (J')| \le C|J|.
$$
As in Corollary~\ref{regular}, consider  $w(t) = \sup \mu(J)$, where the supremum is taken over all arcs $J \subset\mathbb T$ with $|J|= t$. Then there exists a constant $C>0$ such that $w(2^{-n}) \le Cn 2^{-n}$ for any positive integer $n$ and Corollary~\ref{regular} gives that $S_\mu$ is non wepable.
\end{remark}

\section{Atomic Measures and Easily Wepable Singular Functions}
\label{new}

We start with a modification of a construction from the proof of Proposition~6 in \cite{Bo}.

\begin{lemma}
\label{kl6}
Given a large $A$, there exists $m\gg 1$ satisfying the following property: if $J\in\mathcal D_k$, $J$ is the union of $2^m$ arcs $J_s\in \mathcal D_{k+m}$, $x_s\in J_s$, $0\le s<2^m$,  
then there exists a set $Y\subset(x_s)_{0\le s<2^m}$, $\card Y\lesssim 2^m/A^2$, such that if 
$$
\mu=\sum_{y\in Y}c_y\delta_y
$$
with $10\cdot 2^{-k-m}A\le c_y\le 20\cdot 2^{-k-m}A$ (and we have $\|\mu\|\lesssim |J|/A$), 
if $\Lambda$ is a subset of $\mathbb D$ such that
$$
P[\mu](z(I))\ge A \implies \Lambda\cap T(I)\not =\varnothing, \qquad I\in \mathcal D,
$$
and if $B_\Lambda$ is the corresponding Blaschke product, then
$$
\log|B_\Lambda(z(J))|\lesssim -A^2.
$$
\end{lemma}

\begin{proof} Let $J=[0,2\pi \cdot2^{-k})$. Fix $q\in\mathbb N$, $q\ge A^2$, $n=q^2$, and $m\ge 1$ such that $2^{m-1}\le q2^{2n}< 2^m$ and define 
$$
Y=\{x_s:s=jq2^n+\ell, \, 0\le j<2^n,\, 0\le \ell <2^n\}.
$$
Then $\|\mu\|\asymp 2^{-k}A/q\lesssim |J|/A$, and 
\begin{multline*}
2^{-k-m}\le 1-r\le 2^{n-k-m} \,\,\&\,\, 0\le \theta-jq2^{n-k-m}\le 2^{n-k-m}\\ \implies P[\mu](re^{2\pi i\theta})\ge A.
\end{multline*}
Hence, for every $I\in\cup_{k+m-n\le p\le k+m}\mathcal D_p$ such that 
$$
I\subset \cup_{0\le j<2^n}[2\pi \cdot jq2^{n-k-m},2\pi \cdot(jq+1)2^{n-k-m}]
$$
we have
$$
\Lambda\cap T(I)\not =\varnothing.
$$
Thus, applying \eqref{kl8},
$$
\log|B_\Lambda(z(J))|\lesssim -q.
$$
\end{proof}

\begin{proof}[Proof of Theorem~\ref{thmatom}]
$(a)\Rightarrow(b)$ If \eqref{kl5} does not hold, then we can choose a sequence of groups $(b_{s_n},\ldots,b_{s_n+n})_{n\ge 1}$ such that $b_{s_n}\le 2b_{s_n+n}$. 
By Lemma~\ref{kl6}, passing to a subsequence of $(s_n)$ denoted also by $(s_n)$ we construct a sequence of dyadic arcs $J_n$ and measures 
$$
\mu_n=\sum_{s_n\le y<s_n+2^{m_n}} b_y \delta_{x_y} 
$$
such that 
$$
\supp \mu_n\subset J_n,\quad \|\mu_n\|=o(|J_n|), \quad \dist(J_n,J_{n'})\gtrsim \max(|J_n|,|J_{n'}|).
$$
Furthermore, if the sets $\Lambda_n$ satisfy the property 
\begin{equation}
\label{kl11}
P[\mu_n](z(I))\ge n \implies \Lambda_n\cap T(I)\not =\varnothing, \qquad I\in \mathcal D,
\end{equation}
then the corresponding Blaschke products $B_{\Lambda_n}$ satisfy the estimate
$$
|B_{\Lambda_n}(z(J_n))|\le \exp(-n^2).
$$
Finally, we take $x\in\mathbb T\setminus \cup_n J_n$ and set 
$$
\mu=\sum_{n\ge 1} \mu_n+\Bigl(\sum_{s\ge 1}b_s-\sum_{n\ge 1}\|\mu_n\|\Bigr)\delta_x.
$$

Suppose that $B_\Lambda$ is a Blaschke product with zero set $\Lambda$ such that $S_\mu B_\Lambda$ has the WEP. Then for every $n\ge n_0$, 
the set $\Lambda\cap Q(J_n)$ satisfies the property \eqref{kl11}, and hence,
$$
|B_{\Lambda}(z(J_n))|\le \exp(-n^2).
$$
Therefore, $B_{\Lambda}$ should have a zero in $T(J_n)$, $n\ge n_1$. However, 
$$
P[\mu](z(J_n))\to 0,\qquad n\to\infty.
$$
Thus, $S_\mu$ is not easily wepable.

$(b)\Rightarrow(a)$ Given $z,w\in\mathbb D$, set
$$
\beta(z,w)=\log_2\frac{1}{1-\rho(z,w)}.
$$
Let $J_1,J_2\in\mathcal D$, $2^{-n}=|J_1|\le|J_2|=2^{-m}$, $J\in\mathcal D$, $J_1\subset J$, $|J|=|J_2|$, 
$J=[k2^{-m},(k+1)2^{-m}]$, $J_2=[k'2^{-m},(k'+1)2^{-m}]$. Then
$$
\beta_{J_1,J_2}\overset{\rm def}{=}\beta\bigl(z(J_1),z(J_2)\bigr)=n-m+2\log_2(|k-k'|+1)+O(1).
$$
Furthermore, if $T(J_1)\cap T(J_2)=\varnothing$, then 
$$
\min_{z_1\in T(J_1),\,z_2\in T(J_2)}\beta(z,w)=\beta_{J_1,J_2}+O(1).
$$

Let $\phi$ be an increasing subadditive function on $(0,+\infty)$ such that $\phi(x)=x$, $0<x\le 1$, $\phi(x)\asymp \log x$, $x\to\infty$.

Let $\mu=\sum_{s\ge 1}b_s\delta_{x_s}$. For every $J\in\mathcal D$ we set $\lambda(J)=P[\mu](z(J))$. Harnack's principle gives us a Lipschitz type estimate
\begin{equation}\label{kl1}
|\phi(\lambda(J_1))-\phi(\lambda(J_2))|\lesssim \beta_{J_1,J_2}.
\end{equation}
Now, for every $J\in\mathcal D$ we denote by $k_J$ the integer part of $\phi(\lambda(J))$, and 
choose $k_J$ points $z_{J,1},\ldots,z_{J,k_J}$ uniformly distributed in $T(J)$. 
Let $B$ be the Blaschke product with zeros in the points $z_{J,1},\ldots,z_{J,k_J}$, $J\in\mathcal D$. To check that $BS_\mu$ has the WEP 
(and incidentally that $B$ exists) we need only to verify that for every $J\in\mathcal D$ we have
\begin{equation}\label{kl2}
\sum_{I\in\mathcal D\setminus\{J\}} 2^{-\beta_{J,I}}\phi(\lambda(I))\lesssim \max(\lambda(J),1).
\end{equation}

Fix $J\in\mathcal D$. Let $|J|=2^{-n}$. By \eqref{kl1}, we have
\begin{multline*}
\sum_{I\in\mathcal D,\, |I|> |J|} 2^{-\beta_{J,I}}\phi(\lambda(I))\lesssim\sum_{I\in\mathcal D,\, |I|> |J|} 2^{-\beta_{J,I}}
(\phi(\lambda(J))+\beta_{J,I})\\ \lesssim
\sum_{0\le k<n}\sum_{s=1}^{2^k}2^{k-n}s^{-2}(\phi(\lambda(J))+n-k+\log s+O(1))\lesssim \max(\lambda(J),1).
\end{multline*}

Next, we set $\mu'=\chi_{10J}\mu$, $\mu''=\mu-\mu'$, and define $\lambda'(I)=P[\mu'](z(I))$, $\lambda''(I)=P[\mu''](z(I))$, $I\in\mathcal D$.
To prove \eqref{kl2}, we need only to check that
\begin{align}
\sum_{I\in\mathcal D\setminus\{J\},\, |I|\le |J|} 2^{-\beta_{J,I}}\phi(\lambda'(I))&\lesssim \max(\lambda'(J),1),\notag\\
\sum_{I\in\mathcal D\setminus\{J\},\, |I|\le |J|} 2^{-\beta_{J,I}}\phi(\lambda''(I))&\lesssim \max(\lambda''(J),1)\label{kl3}.
\end{align}

We have $\mu'=\sum_{s\in N'}b_s\delta_{x_s}$, and we set $a=\max\{b_s:s\in N'\}$. (If $\mu'=0$, we just pass to $\mu''$.)
By \eqref{kl5}, $P[\mu'](z(J))\asymp 2^na$, and hence, 
$$
\sum_{I\in\mathcal D\setminus\{J\},\, |I|\le |J|} 2^{-\beta_{J,I}}\phi(\lambda'(I))\lesssim
\sum_{s\in N'} \sum_{I\in\mathcal D\setminus\{J\},\, |I|\le |J|}  2^{-\beta_{J,I}}\phi(b_s  P[\delta_{x_s}](z(I))).
$$
Fix for a moment $s\in N'$. Without loss of generality, $x_s=0$, and for $m\ge n$, $I=[2\pi \cdot t2^{-m},2\pi \cdot (t+1)2^{-m})$, $I\neq J$, we have 
$$
\beta_{J,I}\ge n-m+O(1),\qquad P[\delta_{x_s}](z(I))\lesssim 2^mt^{-2}.
$$
Hence, 
\begin{gather*}
\sum_{I\in\mathcal D\setminus\{J\},\, |I|\le |J|} 2^{-\beta_{J,I}}\phi(\lambda'(I))\lesssim
\sum_{s\in N'}\sum_{m\ge n}\sum_{t\ge 1} 2^{n-m}\phi(b_s2^m t^{-2})\\=
\sum_{s\in N'}\sum_{m\ge n}\sum_{t^2\le 2^mb_s} 2^{n-m}\phi(2^mb_s t^{-2})+
\sum_{s\in N'}\sum_{m\ge n}\sum_{t^2> 2^mb_s} 2^{n-m}\phi(2^mb_s t^{-2})\\ \lesssim
\sum_{s\in N'}\sum_{m\ge n}2^{m/2}\sqrt{b_s} 2^{n-m}(\phi(2^mb_s)+1)\\ \lesssim  
\sum_{s\in N'}\sum_{m\ge n} 2^{n-m/2}\sqrt{b_s}\max(m+\log b_s,1)\\ \lesssim \sum_{s\in N'} 2^{n/2}\sqrt{b_s} \max(n+\log b_s,1)\\ \lesssim
2^{n/2}\sqrt{a}\max(n+\log a,1)\lesssim \max(2^na,1)\lesssim \max(\lambda'(J),1).
\end{gather*}

We have $\mu''=\sum_{s\in N''}b_s\delta_{x_s}$. To prove \eqref{kl3}, we need to verify that 
\begin{multline}\label{kl4}
\sum_{s\in N''}\sum_{I\in\mathcal D\setminus\{J\},\, |I|\le |J|} 2^{-\beta_{J,I}}\phi(b_s  P[\delta_{x_s}](z(I)))\\  \lesssim   1+\sum_{s\in N''}P[b_s\delta_{x_s}](z(J)).
\end{multline}

Fix $s\in N''$ and choose $r=r(s)\ge 2$ such that 
$$
2r|J|=\dist(x_s,J).
$$
Then
\begin{gather*}
\sum_{I\in\mathcal D\setminus\{J\},\, |I|\le |J|} 2^{-\beta_{J,I}}\phi(b_s  P[\delta_{x_s}](z(I)))\\=
\sum_{I\in\mathcal D\setminus\{J\},\, |I|\le |J|,\,\dist(I,J)\le r|J|} 2^{-\beta_{J,I}}\phi(b_s  P[\delta_{x_s}](z(I)))\\+
\sum_{I\in\mathcal D\setminus\{J\},\, |I|\le |J|,\,\dist(I,J)> r|J|} 2^{-\beta_{J,I}}\phi(b_s  P[\delta_{x_s}](z(I)))=A_1+A_2.
\end{gather*}
Next, we can assume that $J=[0,2^{-n})$ and for $I=[2\pi\cdot t2^{-m},2\pi\cdot (t+1)2^{-m})$, $m\ge n$, we have $\beta_{J,I}=m-n+2\log_2(1+t2^{n-m})+O(1)$. Hence, 
\begin{multline*}
A_1\lesssim \sum_{m\ge n}\sum_{t\ge 1} 2^{n-m}(1+t2^{n-m})^{-2}\phi(b_s2^{-m} r^{-2}2^{2n})\\ 
\lesssim \sum_{m\ge n}\phi(2^{2n-m}b_s r^{-2})\lesssim 2^nb_sr^{-2}\asymp P[b_s\delta_{x_s}](z(J)).
\end{multline*}
Furthermore,
\begin{gather*}
A_2\lesssim \sum_{m\ge n}\sum_{t\ge 1} 2^{n-m}r^{-2}\phi(b_s2^{m} t^{-2})\\=
\sum_{m\ge n}\sum_{t^2\le 2^mb_s} 2^{n-m}r^{-2}\phi(2^mb_s t^{-2})+
\sum_{m\ge n}\sum_{t^2>2^mb_s} 2^{n-m}r^{-2}\phi(2^mb_s t^{-2})\\ \lesssim 
\sum_{m\ge n} 2^{m/2}\sqrt{b_s} 2^{n-m}r^{-2}\phi(2^mb_s)+
\sum_{m\ge n} 2^{n-m}r^{-2} 2^m b_s2^{-m/2}b_s^{-1/2}\\ \lesssim 2^{n/2}\sqrt{b_s}r^{-2}  \max(n+\log b_s,1)+ 2^{n/2}\sqrt{b_s}r^{-2}.
\end{gather*}
If $b_s\ge 2^{-n}$, then 
$$
A_2\lesssim 2^nb_sr^{-2}\asymp P[b_s\delta_{x_s}](z(J)).
$$
Thus, to complete the proof, we need only to estimate 
$$
H=\sum_{s\in N'',\,b_s<2^{-n}}2^{n/2}\sqrt{b_s}r(s)^{-2}. 
$$
By \eqref{kl5}, 
$$
\sum_{b_s<2^{-n}}\sqrt{b_s}\lesssim 2^{-n/2}.
$$
Hence,
$$
H\lesssim 1,
$$
and \eqref{kl4} follows.
\end{proof}

\section{Porous Sets and Easily Wepable Singular Functions}
\label{proofpor}

\begin{proof}[Proof of Theorem~\ref{thmpor}. $(b)\Rightarrow(a)$]
Suppose that $E$ is porous and set $u= P[\mu]$.
Set
$$
\mathcal G_k := \{ I\in\mathcal D: k^2 < u(z(I)) \le (k+1)^2 \},\qquad k\ge 1.
$$
We claim that there
exists a constant $C>0$  such that
\begin{equation}
\label{mainest}
\sum_{k\ge 1} k \sum_{I\in \mathcal G_k} \frac{|I||J|}{|1- \overline{z(J)} z(I)|^2}
\le C u(z(J)),\qquad J\in\mathcal D.
\end{equation}
Indeed, by Lemma~\ref{lempor}~(c), 
\begin{multline*}
\sum_{k\ge 1} k \sum_{I\in \mathcal G_k} \frac{|I||J|}{|1- \overline{z(J)} z(I)|^2}
\\ \le
\sum_{k\ge 1}  \sum_{I\in\mathcal D,\, k^2 < u(z(I))} \frac{|I||J|}{|1- \overline{z(J)} z(I)|^2}
\le
C \sum_{k\ge 1} \frac{u(z(J))}{k^2}.
\end{multline*}
Now for each integer $k\ge 1$ and each $I \in \mathcal G_k$ we consider the set $\Lambda(I)$ consisting of $k$ points, uniformly distributed in $T(I)$. Let 
$\Lambda:=\cup_{k\ge 1} \cup_{I\in \mathcal G_k}\Lambda(I)$.
Taking $J= \partial \D$ in \eqref{mainest}, we see that $\Lambda$ is a Blaschke sequence. Let
$B$ be the Blaschke product with the zero set $\Lambda$.

Notice that the zeros of $B$ are restricted to the sets $T(I)$ where the
modulus of $S_\mu$ is small, 
so if we prove that $BS_\mu$ has the WEP, we will have shown that
$S_\mu$ is easily wepable.

Furthermore, the zeros of $B$ are more and more densely packed as $k\to\infty$, i.e. as the
modulus of $S_\mu$ gets smaller; thus for any $\eps >0$ there exists  $\eta=\eta(\eps)>0$ such that
$|S_\mu(z)|>\eta$ whenever $\rho(z, \Lambda) > \eps$. Thus, to prove that
$BS_\mu$ has the WEP, we only need to show that $\inf\{ |B(z)|: \rho(z, \Lambda) > \eps\} >0$.
Fix $z$ such that $\rho(z, \Lambda) > \eps$ and let $J\in\mathcal D$ be such that $z \in T(J)$.
Then by Harnack's inequality, 
\begin{equation}
\label{kl18}
u(z(J)) \le C \log \eta^{-1}.
\end{equation}
By \eqref{kl9}, 
\begin{gather*}
\log |B(z)|^{-1} \le
C(\eps) \sum_{\lambda \in \Lambda} \frac{(1-|\lambda|^2)(1-|z|^2)}{|1-\bar \lambda z|^2}\\
=C(\eps) \sum_{k=1}^\infty \sum_{I\in \mathcal G_k} \, \sum_{\lambda \in \Lambda\cap T(I)}
 \frac{(1-|\lambda|^2)(1-|z|^2)}{|1-\bar \lambda z|^2} \\
 \le C(\eps) \sum_{k=1}^\infty k \sum_{I\in \mathcal G_k} \frac{|I||J|}{|1- \overline{z(J)} z(I)|^2} .
\end{gather*}
Applying \eqref{mainest} and \eqref{kl18}, we conclude that $\log |B(z)|^{-1}\le C(\eps) \log \eta^{-1}$.
\end{proof}

\begin{proof}[Proof of Theorem~\ref{thmpor}. $(a)\Rightarrow(b)$]
Assume now that $E$ is not porous. We can find a sequence of arcs $J_n \in \mathcal D_{k_n}$, $k_n\to\infty$, and 
a sequence of numbers $M_n \to\infty$, $n\to\infty$, 
such that every $J\in \mathcal D_{k_n+M_n}$, $J \subset J_n$, meets $E$. Passing to a subsequence and using Lemma~\ref{kl6}
we obtain a sequence of arcs $J_n\in \mathcal D_{k_n}$, $k_n\to\infty$, and a sequence of measures $\mu_n$ such that $\supp \mu_n\subset J_n\cap E$, $\|\mu_n\|=o(|J_n|)$, and  
$\dist(J_n,J_{n'})\gtrsim \max(|J_n|,|J_{n'}|)$.
Furthermore, if sets $\Lambda_n\subset\mathbb D$ satisfy the property 
\begin{equation*}
P[\mu_n](z(I))\ge n \implies \Lambda_n\cap T(I)\not =\varnothing, \qquad I\in \mathcal D,
\end{equation*}
then the corresponding Blaschke products $B_{\Lambda_n}$ satisfy the estimate
$$
|B_{\Lambda_n}(z(J_n))|\le \exp(-n^2).
$$
Let 
$$
\mu=\sum_{n\ge 1} \mu_n.
$$
To conclude that $S_\mu$ is not easily wepable we use the same argument as in the part $(a)\Rightarrow(b)$ of the proof of Theorem~\ref{thmatom}.
\end{proof}

\section{Corona type constants}
\label{proofcsts}

First, we make an easy remark:
$c_n (\delta, I) \ge c_{n-1} (\delta, I)$, thus $\delta_n(I) \ge \delta_{n-1}(I)$.

Indeed, suppose that $\gamma < c_{n-1} (\delta, I)$, then there are $(f_1,\dots,f_{n-1})=:f$
such that $\delta^2 \le \sum_{j=1}^{n-1} |f_j(\lambda)|^2 \le
\|f\|^2_{\infty,n} \le 1$, and that
$$
\gamma < \inf\{ \|g\|_{\infty,n-1} : \exists h\in H^\infty :
\sum_{j=1}^{n-1} g_j f_j + h I \equiv 1\}.
$$
Given $\tilde f:=(f_1,\dots,f_{n-1},0)$, for every $g \in (H^\infty)^n$ we obtain that 
$\sum_{j=1}^n g_j f_j=\sum_{j=1}^{n-1} g_j f_j$, so that 
$\chi_I (\tilde f) \ge \gamma$. Since $\tilde f$ fulfils the condition
to be a candidate in the supremum, we obtain that 
$c_n (\delta, I)\ge \gamma$, q.e.d.

\begin{lemma}
\label{lowerest}
For any $n$, $\delta_n(I) \le \tilde \delta (I)$.
\end{lemma}

\begin{proof}
Pick any number $\eps_0 > \tilde \delta (I)$, then choose $\eps_1$ such that
$\eps_0 > \eps_1 >\tilde \delta (I)$.  Suppose that
$f:=(f_1,\dots,f_n) \in (H^\infty)^n$ satisfies the estimates 
$\eps_0^2 \le \inf_k \sum_{\lambda\in Z(I)}^n |f_j(\lambda) |^2$, $\|f\|_{\infty,n}\le 1$.  Take $z\in \D$
such that for some $\lambda\in Z(I)$ we have 
$\rho(z,\lambda) < \eps_1$. Then, applying the Schwarz-Pick
Lemma to the function $\varphi:= f \cdot \bar v$, where $v$ is a unit
vector in $\mathbb C^n$ parallel to $f(\lambda)$, we see that
$$
\Bigl( \sum_{j=1}^n |f_j(z) |^2\Bigr)^{1/2} \ge |\varphi(z)|\ge \frac{\eps_0-\eps_1}{1- \eps_0\eps_1}=: \eps_2.
$$

On the other hand, suppose that $\rho(z,Z(I)) \ge \eps_1$, then
$|I(z)| \ge \eta_I(\eps_1) >0$. Finally,
$$
\inf_{z\in\D} \left(  \sum_{j=1}^n |f_j(z) |^2 + |I(z)|^2 \right) \ge
\min( \eps_2^2, \eta_I(\eps_1)^2).
$$
By Carleson's Corona Theorem, we can find $g \in (H^\infty)^n$, $h\in H^\infty$ with $\|g\|_{\infty,n}\le C(\eps_2,\eta_I(\eps_1))$ 
such that $\sum_{j=1}^{n} g_j f_j + h I\equiv 1$, therefore $c_n(\eps_0,I)<\infty$.
Since this holds for any $\eps_0 > \tilde \delta (I)$, we are done.
\end{proof}




The following will end the proof of Proposition \ref{deltaseq}.

\begin{lemma}
\label{upperest}
$\delta_1(I)\ge \tilde \delta(I)$.
\end{lemma}

\begin{proof}
Let $\eps_0 < \tilde \delta(I)$. We want to prove that $c_1(\eps_0,I)=\infty$.
Pick $\eps_1$ such that $\eps_0 < \eps_1 < \tilde \delta(I)$.  Then there exists
an infinite sequence $(\zeta_n)_n\subset \D$ such that $\rho(\zeta_n, Z(I))\ge
\eps_1$ and $|I(\zeta_n)| \to 0$.

Choose a subsequence $(\xi_n)$ of this sequence, with
$$
1- \inf_k \prod_{j:j\neq k} \rho (\xi_j,\xi_k)  
$$
so small that the Blaschke product $B$ with zeros $(\xi_n)$ satisfies the property $|B(z)| > \varepsilon_0$ if $\rho(z, Z(B)) > \varepsilon_1$ (see, for instance \cite[p. 395]{Ga}). 
Then for any $\lambda\in Z(I)$ we have $|B(\lambda)| \ge \varepsilon_0$. On the other hand, 
for 
any $g,h \in H^\infty$,
$$
g(\xi_n)B(\xi_n) + h(\xi_n) I(\xi_n) = h(\xi_n) I(\xi_n)
\to 0,\qquad n\to\infty.
$$
This proves that $gf+hI \not \equiv 1$.
\end{proof}

\begin{proof}[Proof of Proposition~\ref{cn}] The argument is anologous to that in the proof of Lemma~\ref{upperest}. Take a strictly increasing function $\psi:(0,1)\to(0,1)$ 
such that $\psi\cdot(\phi+1)\le 1$. Using the above mentioned result from \cite[p. 1199]{Bo} we find a Blaschke product $B$ satisfying the WEP and such that for every $\delta\in(0,1)$ 
there exists $z_\delta\in\mathbb D$ satisfying  
$$
\rho(z_\delta,Z(B))=\delta,\qquad |B(z_\delta)|\le \psi(\delta).
$$
Denote $b_\delta(z)=(z-z_\delta)/(1-\bar z_\delta z)$. We have $\min_{Z(B)}|b_\delta|=\delta$. 
If $g,h \in H^\infty$, $gb_\delta+hB\equiv1$, then 
$$
\|h\|_\infty\ge 1+\phi(\delta),
$$
and hence, 
$$
c_1(\delta,B)\ge \inf_{gb_\delta+hB\equiv1}\|g\|_\infty\ge \phi(\delta),\qquad\delta\in(0,1).
$$
\end{proof}


\end{document}